\documentclass[12pt,reqno]{amsart}
\usepackage[top=1in, bottom=1in, left=1in, right=1in]{geometry}
\usepackage{graphicx}
\usepackage{epstopdf}
\usepackage{verbatim}
\usepackage{enumerate}

\usepackage{amssymb,mathrsfs,amsmath}
\def\Z{\mathbb Z}

\def\F{\mathbb F}

\newcommand{\leg}[2]{\left(\frac{#1}{#2}\right)}

\newtheorem{thm}{Theorem}[section]
\newtheorem{cor}[thm]{Corollary}

\newtheorem{lma}[thm]{Lemma}
\newtheorem{conj}[thm]{Conjecture}
\theoremstyle{remark}
\newtheorem{rmk}[thm]{Remark}

\newcommand {\SZ} {{\mathbb Z}}

\newcommand{\bsp}{\begin{split}}
\newcommand{\esp}{\end{split}}
\newcommand{\be}{\begin{equation}}
\newcommand{\ee}{\end{equation}}
\newcommand{\bes}{\begin{equation*}}
\newcommand{\ees}{\end{equation*}}

\newcommand{\ba}{\begin{align}}
\newcommand{\ea}{\end{align}}
\newcommand{\bas}{\begin{align*}}
\newcommand{\eas}{\end{align*}}

\newcommand{\bv}\boldsymbol{}

\DeclareMathOperator{\meas}{meas}

\numberwithin{equation}{section}

\renewcommand{\pmod}[1]{\,(\text{mod}\,#1)}

\title{Group structures of elliptic curves over finite fields}

\author[V. Chandee]{Vorrapan Chandee}
\address[Vorrapan Chandee]{
Centre de recherches math\'ematiques\\
Universit\'e de Montr\'eal\\
P.O. Box 6128\\
Centre-ville Station\\
Montr\'eal, Qu\'ebec\\
H3C 3J7\\
Canada;
Department of Mathematics \\
Burapha University \\
169 Long-Hard Bangsaen Rd \\
Chonburi, 20131 \\
Thailand
}
\author[C. David]{Chantal David}

\address[Chantal David]{
Department of Mathematics and Statistics\\
Concordia University\\
1455 de Maisonneuve West\\
Montr\'eal, Qu\'ebec\\
H3G 1M8\\
Canada
}
\author[D. Koukoulopoulos]{Dimitris Koukoulopoulos}
\address[Dimitris Koukoulopoulos]{
D\'epartement de math\'ematiques et de statistique\\
Universit\'e de Montr\'eal\\
CP 6128, Succ. Centre-Ville\\
Montr\'eal, QC H3C 3J7
}
\author[E. Smith]{Ethan Smith}
\address[Ethan Smith]{
Department of Mathematics\\
Liberty University\\
1971 University Blvd\\
MSC Box 710052\\
Lynchburg, VA  24502
}

\begin{document}

\maketitle

\begin{abstract}
It is well-known that if $E$ is an elliptic curve over the finite field $\F_p$, then $E(\F_p)\simeq\Z/m\Z\times\Z/mk\Z$ for some positive integers $m, k$.
Let $S(M,K)$ denote the set of pairs $(m,k)$ with $m\le M$ and $k\le K$ such that there exists an elliptic curve over some prime finite field whose
group of points is isomorphic to $\Z/m\Z\times\Z/mk\Z$.  Banks, Pappalardi and Shparlinski recently conjectured that if $K\le (\log M)^{2-\epsilon}$,
then a density zero proportion of the groups in question actually arise as the group of points on some elliptic curve over some prime finite field.  On the
other hand, if $K\ge (\log M)^{2+\epsilon}$, they conjectured that a density one proportion of the groups in question arise as the group of
points on some elliptic curve over some prime finite field.  We prove that the first part of their conjecture holds in the full range
$K\le (\log M)^{2-\epsilon}$, and we prove that the second part of their conjecture holds in the limited range $K\ge M^{4+\epsilon}$.  In the wider range
$K\ge M^2$, we show that at least a positive density of the groups in question actually occur.
\end{abstract}

\section{Introduction}

Let $E$ be an elliptic curve over $\F_p$, and denote with $E(\F_p)$ its set of points over $\F_p$. It is well-known that $E(\F_p)$ admits the structure of an abelian group. It is then natural to ask for a description of the groups that arise this way as $p$ runs through all primes and $E$ through all curves over $\F_p$. This question was first addressed by Banks, Pappalardi and Shparlinski in~\cite{BPS}. Below we reproduce part of the discussion from~\cite{BPS}.

The first relevant property is that the size of $E(\F_p)$ can never be very far from $p+1$.
Indeed, if $\# E(\F_p) = p + 1 - a_p$, then Hasse proved that $|a_p| \leq 2 \sqrt{p}$.
Setting
\[
x^{-} = x+1-2\sqrt{x}=(\sqrt{x}-1)^2
	\quad\text{and}\quad
x^{+}= x+1+2\sqrt{x}=(\sqrt{x}+1)^2.
\]
for each $x\ge 1$, this is equivalent to saying that $\#E(\F_p)\in (p^-,p^+)$.
It follows from the work of Deuring~\cite{Deu} that for any integer $N$ satisfying $p^-<N<p^+$, there exists an
elliptic curve $E/\F_p$ with $\# E(\F_p) = N$.
Solving the inequalities for $p$ allows us to conclude that, given a positive integer $N$, there is a finite field $\F_p$ and an elliptic curve $E/\F_p$ with $\#E(\F_p)=N$ if and only if there is a prime
$p\in (N^-,N^+)$.
However, this result does not take into account the actual group structure of $E(\F_p)$.

The second relevant property is that, as an abstract abelian group, $E(\F_p)$ has at most two invariant factors.
In other words, we may write that
\[
E(\F_p) \simeq G_{m,k} := \Z / m \Z \times \Z / m k \Z
\]
for some unique positive integers $m,k$.  Refining the ideas already present in the work of Deuring, one can argue that there is an elliptic curve $E/\F_p$ with $E(\F_p)\simeq G_{m,k}$ if and only if $N=m^2k\in (p^-,p^+)$ and
$p\equiv 1\pmod m$.  Arguing as before allows us to conclude that, given a group $G_{m,k}$ of order $N=m^2k$, there is a finite field $\F_p$ and an elliptic curve $E/\F_p$ with $E(\F_p)\simeq G_{m,k}$ if and only if
there is a prime $p\equiv 1\pmod m$ in the interval $(N^-,N^+)$.  The latter condition is equivalent to the assertion that
there is a prime of the form $p = km^2 + j m  + 1$ with $|j|< 2 \sqrt{k}$.  See Corollary~\ref{basicprop} below.

The above characterization gives some interesting consequences. Note that when $k$ is very small it is unlikely that there is a finite field $\F_p$ and a curve $E / \F_p$ such that $E(\F_p) \simeq G_{m,k}$ simply because the interval $(N^-,N^+)$ is too short. For example, there is no curve over $\F_p$ such that
$E(\F_p) \simeq \Z / 11 \Z \times \Z / 11 \Z$, since none of the three integers
$122-11, 122, 122+11$ is prime. Other examples of groups not occurring are given by Banks, Pappalardi, and Shparlinski in~\cite{BPS}.

In order to study the question of which groups $G_{m,k}$ occur as group structures of elliptic curves over $\F_p$ from an average point of view, the authors of~\cite{BPS} defined
\[
S(M, K) = \left\{ m \leq M,k \leq K : \mbox{there is a prime $p$ and a curve $E / \F_p$
with $E(\F_p) \simeq G_{m,k}$} \right\}.
\]
They proved the following result for the cardinality of $S(M,K)$.

\begin{thm}[Banks, Pappalardi, and Shparlinski~\cite{BPS}]\label{theoremBPS} Let $M\ge2$ and $K\ge1$. Then for every fixed $K$, we have
\[
\# S(M,K) \ll_K \frac{M}{\log{M}}.
\]
If $M \leq K^{43/94-\epsilon}$, then
\[
\# S(M, K) \gg \frac{M K}{\log{K}}.
\]
Finally, if $M \leq K^{1/2 - \epsilon},$ then
\[
\# S(M, K) \gg \frac{M K}{(\log{K})^2}.
\]
\end{thm}

Moreover, the authors of~\cite{BPS} conjectured the following.

\begin{conj}[Banks, Pappalardi, Shparlinski ~\cite{BPS}]\label{conjBPS}
\[
\# S(M, K) =
	\begin{cases}
		{o}(M K) 			& \mbox{if $K \le (\log{M})^{2 - \epsilon}$}, \cr
 		MK ( 1 + {o}(1) )  	& \mbox{if $K \ge (\log{M})^{2 + \epsilon}$.}
	\end{cases}
\]
\end{conj}

The motivation behind the above conjecture can be explained by a simple heuristic. An integer $n$ is prime with probability
about $1/\log n$. For $G_{m,k}$ to be the group of a curve $E$ over some finite field, we need at least one of the integers $n = km^2 + j m + 1$ with  $|j| < 2 \sqrt{k}$ to be prime.
If we assume that these events occur independently of each other, the probability that none of the integers $n = km^2 + j m + 1$, $|j| < 2 \sqrt{k}$, is prime is about
$$
\left(1-\frac1{\log(m^2k)}\right)^{4\sqrt{k}}.
$$
This quantity becomes less than one as soon as $\sqrt{k}\gg\log(m^2k)$. In particular, if $k \ge (\log m)^{2+\epsilon}$, then we expect with probability 1 that
$km^2 + j m + 1$ is prime for some $j\in(-2\sqrt{k}, 2 \sqrt{k})$.  One can make the even bolder guess that if $k$ is large enough, then there is always some $j\in(-2\sqrt{k}, 2 \sqrt{k})$ for which $km^2 + j m + 1$ is prime. This question is completely out of reach with the current technology, as we do not even know whether there are primes in every interval of the form $(x,x+x^{0.524})$ with $x$ large enough.\footnote{The best result known, due to Baker, Harman and Pintz~\cite{BHP}, is that $(x,x+x^{0.525})$ contains primes for every sufficiently large $x$.}

In this paper we improve upon Theorem~\ref{theoremBPS}. Our first result is that the first part of Conjecture~\ref{conjBPS} holds for $M,K$ in the predicted range.

\begin{thm} \label{thm-ksmall} Let $M\ge2$ and $K\ge1$. Then we have that
\[
\# S(M,K) \ll \frac{MK^{3/2}}{\log M}.
\]
In particular, if $K\le(\log M)^{2-\epsilon}$ for some fixed $\epsilon>0$, then
\[
 \# S(M, K) = o_\epsilon(MK)\quad\text{as}\ M\to\infty.
\]
\end{thm}

\begin{rmk}
The proof of Theorem~\ref{thm-ksmall} begins in the same way as the proof given in~\cite{BPS} for the first assertion of
Theorem~\ref{theoremBPS}.  The main difference is that we explicitly calculate the dependence on $K$.
Calculation of this dependence relies on the ability to approximate the value of $L(1,\chi)$ by a very short Euler product
for most characters $\chi$.  See Lemma~\ref{lemma-Elliot} below.
\end{rmk}

We also prove that the second part of Conjecture \ref{conjBPS} holds for a restricted range of
$M$ and $K$.

\begin{thm} \label{Kasymp} Fix $A\ge1$ and $\epsilon\in(0,1/3]$.
If $M \leq K^{1/4-\epsilon}$, then
\[
 \# S(M, K) = MK+O_{\epsilon,A}\left( \frac{MK}{(\log K)^A}\right).
\]
If, in addition, the Riemann hypothesis for Dirichlet $L$-functions is true, then the above estimate holds when $M\le K^{1/2-\epsilon}$.
\end{thm}

Finally, we show that a lower bound of the correct order of magnitude holds unconditionally in the range $M\le K^{1/2}$:

\begin{thm} \label{Kbiglower}
For $1\le M\le K^{1/2}$, we have that
\begin{equation*}
\# S(M, K) \gg MK.
\end{equation*}
\end{thm}

\subsection*{Notation} Given an integer $n$, we let $P^+(n)$ and $P^-(n)$ denote its largest and smallest primes factors, respectively, with the notational conventions that $P^+(1)=1$ and $P^-(1)=\infty$. As usually, $\tau,\ \mu$, $\phi$ and $\Lambda$ denote the divisor, the M\"obius, the totient and the von Mangoldt function, respectively. Furthermore, we let $\pi(x;q,a)$ be the number of primes up to $x$ that are congruent to $a\pmod q$ and
\[
\psi(x;q,a) = \sum_{n\equiv a\pmod q} \Lambda(n).
\]
The letters $p$ and $\ell$ always denote prime numbers. Finally, we write $f\ll_{a,b,\dots}g$ if there is a constant $c$, depending at most on $a,b,\dots$, such that $|f|\le cg$, and we write $f\asymp_{a,b,\dots}g$ if $f\ll_{a,b,\dots}g$ and $g\ll_{a,b,\dots} f$.

\section{Preliminaries and Cohen-Lenstra heuristics}\label{CLH}

In this section we explain how the existence of an elliptic curve over a prime finite field with a given group structure
is equivalent to the existence of a prime in a certain interval with a given congruence condition.
%In this section we explain how counting the number of curves over finite fields with a given group structure is equivalent to counting primes in intervals with given congruences.
Some of the results and arguments of this section are very similar to Section 3 of \cite{BPS}, but we reproduce them here for the sake of completeness. The first lemma is a result of R\"uck \cite{ruck}, who used the work of Deuring, Waterhouse, and Tate-Honda to characterize those groups which actually occur as the group of points on
elliptic curves over finite fields.

\begin{lma}[R\"uck]\label{lemma-ruck}
Let $N=\prod_{\ell} \ell^{h_\ell}$ be a possible order $\#E(\F_p)$ for an elliptic curve $E/\F_p$, i.e.,
$N \in (p^{-}, p^+)$. Then all the possible groups $E(\F_p)$ with $\# E(\F_p) = N$ are
$$
\Z / p^{h_p}\Z \times \prod_{\ell \neq p}\left( \Z / \ell^{b_\ell} \Z \times \Z / \ell^{h_\ell-b_\ell}\Z\right)
$$
where $b_\ell$ are arbitrary integers satisfying $0 \leq b_\ell \leq \min{ \left( v_\ell(p-1), \lfloor \frac{h_\ell}{2} \rfloor \right)},$
and $v_\ell(\alpha)$ is the highest power of $\ell$ dividing $\alpha.$
\end{lma}

As a corollary of the above lemma, we have the following result, which is Lemma 3.5 in~\cite{BPS}.

\begin{cor}\label{basicprop}  Let $m$ and $k$ be integers. There is a prime $p$ and a curve $E$ over $\F_p$ such that $E(\F_p) \simeq G_{m,k}$ if and only if there is a prime $p \equiv 1 \pmod m$ in
the interval  $$I_{m^2k} :=  \left( km^2 - 2m \sqrt{k}+1, km^2 + 2m \sqrt{k}+1 \right)$$
or, equivalently, if and only if there is a prime
$p = km^2 + j m + 1$ with $|j|< 2 \sqrt{k}$.
\end{cor}

\begin{proof}
Suppose that there exists an elliptic curve $E$ over $\F_p$ such that $E(\F_p) \simeq G_{m,k}$.
As mentioned in the introduction, we must have that $N=m^2k=\#E(\F_p)\in (p^-,p^+)$.
Solving for $p$ as in the introduction gives that $p\in (N^-, N^+) = I_{mk^2}$.
Since the $m$-torsion points are contained in $E(\F_p)$ and since the Weil pairing is surjective,
$\F_p$ must contain the $m$-th roots of unity, which is equivalent to saying that $p \equiv 1 \pmod m$.
%It turns out that this condition is also sufficient: there is a curve $E$ over $\F_p$ such that $E(\F_p) \simeq G_{m,k}$ if and only if $N=m^2k \in (p^-, p^+)$ and $p \equiv 1 \pmod m$. Equivalently, there is a curve $E$ over $\F_p$ such that $E(\F_p) \simeq G_{m,k}$ if and only if there is a prime $p \equiv 1\pmod m$ that belongs to $(N^-, N^+) = (km^2 - 2 m \sqrt{k} + 1, km^2 + 2 m \sqrt{k} + 1)$,
%i.e. a prime $p = km^2 + j m  + 1$ with $|j|< 2 \sqrt{k}$.

%As we mentioned in the introduction, if $E(\F_p) \simeq G_{m,k}$ for some prime $p$ and some elliptic curve $E/\F_p$, then
%the Hasse bound and the Weil pairing imply that $p \in I_{m^2k}$ and $p \equiv 1 \pmod m$, respectively.
Conversely, suppose that there is a prime $p \in I_{m^2k}$ such that $p \equiv 1 \pmod m$, and let $N=km^2$.
It is easy to check that $|p+1-N| \leq 2 \sqrt{p}$, that is to say that $N$ is an admissible order.
Writing $N = km^2 = \prod_{\ell} \ell ^{h_\ell}$, we clearly have that $v_\ell(m) \leq\lfloor h_\ell/2 \rfloor$.
Furthermore, since $p \equiv 1 \pmod m$, we also have that $v_\ell(p-1) \geq v_\ell(m)$ for each $\ell \mid m$.
Thus, we may take $b_\ell=v_\ell(m)$ in Lemma \ref{lemma-ruck} for all $\ell\mid m$.
So, in particular, $h_\ell - b_\ell = v_\ell(m) + v_\ell(k)$, and we conclude that
\[
G_{m,k} = \prod_{\ell}\left( \Z / \ell^{v_\ell(m)} \Z \times \Z / \ell^{v_\ell(m)+v_\ell(k)} \Z\right)
\]
is an admissible group. This completes the proof of the corollary.
\end{proof}

\begin{rmk}\label{rem:aftercor2.2}
Using Corollary \ref{basicprop}, we readily find that
\[
\# S(M, K) = \sum_{m \leq M} \sum_{k \leq K} \mathbb{I} (m,k),
\]
where
\[
\mathbb{I}(m,k) :=
		\begin{cases}
				1 & \mbox{if there exists a prime $p \in I_{m^2k}$ such that $p \equiv 1 \pmod m$,} \cr
				0 & \mbox{otherwise.} \end{cases}
\]
\end{rmk}

The fact that the groups $G_{m,k}$ are more likely to occur when $m$ is small
%can be seen using
is in accordance with the general philosophy of the Cohen-Lenstra heuristics, which predict that random abelian
groups ``naturally" occur with probability inversely proportional to the size of their automorphism groups.
% $G$ occurs with probability weighted by $\# G / \# \mbox{Aut}(G).$
That is, the presence of many automorphisms decreases the frequency of occurrence.
In particular, those groups which are ``nearly cyclic" ($m$ relatively small) should be the most likely to occur,
and those groups which are ``very split" ($m$ relatively large) should be the least likely to occur.
This is what we observe in Theorems~\ref{thm-ksmall},~\ref{Kasymp}, and~\ref{Kbiglower}.
Indeed, the ``very split" groups occur with density zero, and the ``nearly cyclic" groups occur with density one.
% (and those groups even occur with density zero when they are too split as we proved in Theorem \ref{thm-ksmall}).

In order to see that the probability of occurrence of the groups $G_{m,k}$ is really in correspondence with the weights suggested by the Cohen-Lenstra heuristics, one should count the number of times a given group $G_{m,k}$ occurs as $E(\F_p)$, and not only if it occurs.
More precisely, given a group $G$ of order $N$ and a prime $p$, let
\[
M_p(G) = \# \left\{ E / \F_p : E(\F_p) \simeq G \right\}.
\]
The quantity in question then is the sum
\[
M(G) := \sum_{N^- < p < N^+ } M_p(G).
\]
Using the proper generalization of Deuring's work, $M(G)$ can be related to a certain average of Kronecker class numbers.
See~\cite{Sch} for example.
It is shown in~\cite{David-Smith-MEG} that, under a suitable hypothesis for the number of primes in short arithmetic progressions,
\begin{eqnarray} \label{CH}
\frac{M(G_{m,k})}{4 \sqrt{N}/\log N} \sim_A
K(G_{m,k})\cdot \frac{\# G_{m,k}}{\# \mbox{Aut}(G_{m,k})} \cdot N^{3/2}\quad (N=m^2k,\,m\le(\log k)^A,\,k\to\infty),
\end{eqnarray}
where $K(G_{m,k})$ is non-zero and uniformly bounded for all integers $m$ and $k$. So we see that the average frequency of occurrence of groups of elliptic curves over finite fields is compatible  with the Cohen-Lenstra heuristics.
%%%%That is to say that the presence of many automorphisms decreases the frequency of occurrence.

As we mentioned above, the results of \cite{David-Smith-MEG} are conditional under some hypothesis for the number of primes
in short arithmetic progressions because the intervals $(N^-, N^+)$ are so short that
even the Riemann hypothesis does not guarantee the existence of a prime. Nevertheless, it is possible to obtain unconditional
results displaying the Cohen-Lenstra phenomenon, by showing that the asymptotic in~\eqref{CH}
is an upper bound for all groups $G$, and a lower bound for most of the groups $G$ (modulo constants). This work is in progress \cite{CDKS2}.
The proof of the lower bound for most of the groups $G$ has similarities with the proof of Theorem \ref{Kasymp} of the present paper and, in particular, it requires the generalization of Selberg's theorem about primes in short arithmetic progressions due to the third author
\cite{DK}, but it involves more technical difficulties, as one needs to combine this with the arguments of \cite{David-Smith-MEG}.

\section{Auxiliary results}

In this section, we collect some technical results that will be needed to prove the theorems. First, we state the fundamental lemma of the combinatorial sieve (see, for example,~\cite[Theorem 3, p. 60]{Te}), which will be used in the proof of Theorem \ref{thm-ksmall}. Given a finite set of integers $\mathcal A$ and a number $y\ge1$, we set
\[
S(\mathcal{A},y) := \# \{ a \in \mathcal{A} : P^-(a)>y \}.
\]
As is customary, we assume that there is a multiplicative function $\rho$ and a number $X$ such that for every integer $d$
\[
\#\{a\in\mathcal{A}: a\equiv 0\pmod d\} = X\cdot \frac{\rho(d)}{d}+R_d
\]
for some real number $R_d$, which we think of as an error term. Then we have the following result.

\begin{lma} \label{fundlemma} Let $\mathcal{A}$, $\rho$, $X$ and $\{R_d:d\in\mathbb{N}\}$ be as above. If $\rho(p)\le\min\{2,p-1\}$ for all primes $p$, then we have that
$$ S(\mathcal A, y) = X \prod_{\ell \leq y} \left(1 - \frac{\rho(\ell)}{\ell}\right) \left\{1 + O(u^{-u/2}) \right\} + O\left(\sum_{ \substack{ d \leq y^u ,\, P^+(d)\le y}} \mu^2(d)|R_d| \right),$$ uniformly for all $y\ge1$ and $u\ge1$.
\end{lma}

The next lemma will be used in the proof of Theorem~\ref{thm-ksmall}.

\begin{lma} \label{lemma-truncated}
Fix $\epsilon>0$ and let $\chi$ be a non-principal character mod $q$. For every $y\ge1$, we have that
$$
\prod_{\ell\le y}\left(1-\frac{\chi(\ell)}\ell\right) \ll_\epsilon q^{1/2+\epsilon}.
$$
\end{lma}

\begin{proof} Mertens's estimate implies that
\[
\prod_{\ell\le y}\left(1-\frac{\chi(\ell)}\ell\right) \ll q^{1/2+\epsilon}\prod_{\exp\{q^{1/2+\epsilon}\} <\ell\le y}\left(1-\frac{\chi(\ell)}\ell\right).
\]
Moreover, by the discussion in~\cite[p. 123]{Da}, we have that
\begin{equation}\label{pntap}
\sum_{n\le x}\Lambda(n) \chi(n) \ll_\epsilon\frac x{\log x}\quad(x\ge\exp\{q^{1/2+\epsilon}\}),
\end{equation}
using the trivial bound $\beta < 1 - c/(q^{1/2} \log{q})$ for the Siegel zero
provided by the class number formula.
Partial summation then implies that
\begin{align*}
\log\left\{\prod_{ \exp\{q^{1/2+\epsilon}\} <\ell\le y}\left(1-\frac{\chi(\ell)}{\ell}\right)\right\}
&= -\sum_{\substack{n>1 \\ \ell|n \, \Rightarrow\, \exp\{q^{1/2+\epsilon}\}<\ell \le y}} \frac{\Lambda(n)\chi(n)}{n\log n}  \\
&= -\sum_{\exp\{q^{1/2+\epsilon}\} < n \le y} \frac{\Lambda(n)\chi(n)}{n\log n} +O(1) \ll 1,
\end{align*}
which completes the proof of the lemma.
\end{proof}

The next lemma, which is essentially due to Elliott, allows us to bound the value of $L(1, \chi)$ by a very short product for most quadratic characters $\chi$.

\begin{lma} \label{lemma-Elliot}
Let $\delta\in(0,1]$ and $Q\ge3$. There is a set $\mathcal{E}_\delta(Q)\subset\Z\cap[1,Q]$ of size $ \ll Q^{\delta}$  such that if $\chi$ is a non-principal, quadratic Dirichlet character modulo some $q\le Q$ and of conductor not in $\mathcal{E}_\delta(Q)$, then
$$
\prod_{y<\ell\le z}\left(1 - \frac{\chi(\ell)}{\ell}\right)\asymp_\delta 1 \quad(z\ge y\ge\sqrt{\log Q}).
$$
\end{lma}

\begin{proof} We borrow from the proof of Proposition 2.2 in~\cite{GS}, which is essentially due to Elliott. Without loss of generality, we may assume that $Q$ is large enough. By Theorem 1 in~\cite{Mont}, for every $\sigma_0\in[4/5,1]$, $Q\ge2$ and $T\ge1$, there are $\ll (Q^2T)^{2(1-\sigma_0)/\sigma_0}(\log Q)^{14}$ primitive characters of conductor below $Q$ whose $L$-function has a zero in the region $\{s=\sigma+it\in\mathbb{C}:\sigma\ge\sigma_0,\,|t|\le T\}$. Let $\mathcal{E}_\delta(Q)$ be the set of conductors corresponding to these exceptional characters with $\sigma_0=1-\delta/12\ge11/12$ and $T=Q^3$. If $\chi$ is a Dirichlet character mod $q\in[1,Q]$ of conductor not in $\mathcal{E}_\delta(Q)$, then $L(s,\chi)$ has no zeroes in $\{s=\sigma+it\in\mathbb{C}:\sigma\ge1-\delta/12,\,|t|\le Q^3\}$. So by~\cite[eqn. (17), p. 120]{Da} applied with $T=\min\{Q^3,x\}$, we find that
$$
\sum_{n\le x}\Lambda(n)\chi(n)\ll x^{1-\delta/12}\log^2x +\frac{x\log^2x}{Q^3}+\log^2Q \ll_\delta \frac x{\log x}+\log^2Q\quad(2\le x\le e^{Q}).
$$
The above estimate also holds for $x\ge e^Q$ by~\eqref{pntap}. Together with partial summation, this implies that
\begin{equation} \label{justproven}\begin{split}
\log\left\{\prod_{y<\ell\le z}\left(1-\frac{\chi(\ell)}{\ell}\right)\right\}
&= -\sum_{\ell|n\, \Rightarrow \, y<\ell\le z}\frac{\Lambda(n)\chi(n)}{n\log n}
= -\sum_{y< n\le z}\frac{\Lambda(n)\chi(n)}{n\log n} +O(1) \ll_\delta 1
\end{split}\end{equation}
for $z\ge y\ge\log^2Q$, that is to say, the lemma does hold in this range of $y$ and $z$. Finally, if $\sqrt{\log{Q}} \leq y < \log^2 Q$, then setting $w = \min\{z, \log^2Q\}$, we have
that
\begin{align*}
\prod_{y<\ell\le z}\left(1-\frac{\chi(\ell)}{\ell}\right)
&= \prod_{y<\ell\le w}\left(1-\frac{\chi(\ell)}{\ell}\right)
   \prod_{w<\ell\le z}\left(1-\frac{\chi(\ell)}{\ell}\right) \asymp_\delta 1
\end{align*}
by~\eqref{justproven} and Mertens's estimate, and the lemma follows.
\end{proof}

Next, we state the Bombieri-Vinogradov theorem \cite{Bombieri, Vinogradov1, Vinogradov2}, which will be used to prove Theorem \ref{Kbiglower}.
\begin{lma}[Bombieri-Vinogradov]\label{lemma-BVthm}
Let $A > 0$ be fixed. Then there exists a $B=B(A)>0$, depending on $A$, such that
\begin{equation*}
\sum_{q \leq x^{1/2}/(\log x)^{B}}  \max_{\substack{y\le x \\ (a,q)=1}} \left| \pi(y; q,a) - \frac{{\rm li}(y)}{\phi(q)}\right| \ll \frac{x}{(\log x)^A}.
\end{equation*}
\end{lma}

Finally, in order to prove Theorem~\ref{Kasymp}, we need the following short interval version of the Bombieri-Vinogradov theorem, due to the third author~\cite{DK}.

\begin{lma}\label{average bv} Fix $\epsilon>0$ and $A\ge1$. For $x\ge h\ge2$ and $1\le Q^2\le h/x^{1/6+\epsilon}$, we have that
\[
\int_x^{2x}\sum_{q\le Q}\max_{(a,q)=1}\left|\psi(y+h;q,a)-\psi(y;q,a)-\frac{h}{\phi(q)}\right|dy\ll\frac{xh}{(\log x)^A}.
\]
If, in addition, the Riemann hypothesis for Dirichlet $L$-functions is true, then the above estimate holds when $1\le Q^2\le h/x^\epsilon$.
\end{lma}

\section{Proof of Theorem \ref{thm-ksmall}}

By Remark\ \ref{rem:aftercor2.2} we readily have that
\begin{equation}\label{eqn:trivialbound}
\# S(M, K) \le \sum_{k \leq K} \sum_{|j| < 2 \sqrt{k}} S_{k,j},
\end{equation}
where
$$
S_{k,j} :=
\# \left\{ m \leq M : km^2 + j m + 1 \;\mbox{is  prime} \right\} .
$$
Using the combinatorial sieve to bound  $S_{k,j}$, one immediately obtains as in~\cite{BPS}
that, for any fixed $K$, $\# S(M, K) \ll_K M/\log{M}$. Keeping track of the dependence on $j$ and $k$ of the upper bound for $S_{j,k}$,
we prove Theorem \ref{thm-ksmall}.

In the notation of Lemma \ref{fundlemma}, let $\mathcal{A} = \left\{ km^2 + j m + 1: m \leq M \right\}$, and note that
\begin{align*}
\#\{a\in \mathcal A:a\equiv 0\pmod d\} & = \# \{m \leq M : km^2 + j m + 1 \  \equiv 0 \  \pmod d \} \\
    & = M \cdot \frac{\rho_{k,j}(d)}{d} + O \left( \rho_{k,j}(d) \right),
\end{align*}
where
\[
\rho_{k,j}(d) := \# \left\{ c \in \Z/d\Z : kc^2 + j c + 1 \equiv 0 \pmod d \right\}.
\]
The Chinese remainder theorem implies that $\rho_{j, k}$ is a multiplicative function. Moreover, by a straightforward computation, we find that
\[
\rho_{k,j}(\ell)
	=  \begin{cases}
    		\leg{k - j}{2}^2 & \mbox{if $\ell = 2$}, \cr
 		   1 + \leg{j^2-4k}{\ell}      & \mbox{if $\ell \nmid k$ and $\ell \neq 2$}, \cr
    			\leg{j^2}{\ell}            & \mbox{if $\ell \mid k$},
    	\end{cases}
\]
for all primes $\ell$. Since $S_{k,j} \leq S(\mathcal{A},y) + y$ for all $y$, applying Lemma \ref{fundlemma} with $y = M^{1/2}$ and $u=1$ yields the estimate
\begin{align*}
S_{k,j} & \ll M\prod_{\ell\le y}\left(1-\frac{\rho_{k,j}(\ell)}\ell\right) + \sum_{ d \leq M^{1/2} }\mu^2(d)|\rho_{k,j}(d)| + M^{1/2} \\
&\ll M\prod_{\ell|k,\,\ell\le y}\left(1-\frac{\leg{j^2}\ell}\ell\right)
\prod_{\ell\nmid k,\,\ell\le y} \left(1-\frac{1+\leg{j^2-4k}\ell}\ell \right) + M^{1/2}\log M \\
&\ll \frac M{\log M} \cdot \frac k{\phi(k)} \cdot \prod_{\ell\le y}\left(1-\frac{\leg{j^2-4k}{\ell}}\ell\right)+ M^{1/2}\log M.
\end{align*}
This implies that
\[
\# S(M,K) \ll \frac M{\log M}\sum_{k\le K,\, j < 2\sqrt{k}} \frac k{\phi(k)} \prod_{\ell \le y}\left(1-\frac{\left(\frac{j^2-4k}\ell\right)}\ell\right)
+ M^{1/2}K^{3/2}\log M.
\]
%In the proof of the first statement of Theorem \ref{theoremBPS} in \cite{BPS}, Banks. Pappalardi and Shparlinski did not calculate the sum over $k,$ and then they concluded that $\# S(M,K) \ll_K \frac M{\log M}.$ Here we compute an bound for sum over $k$.

Observing that $j^2-4k\in[-4K,-1]$ for $j$ and $k$ as above, we fix $d\in[1,4K]$ and seek a bound for the sum
\[
T_d := \sum_{\substack{k\le K,\, |j| < 2\sqrt{k} \\ j^2-4k=-d}} \frac k{\phi(k)}
\asymp \sum_{\substack{k\le K,\, |j| < 2\sqrt{k} \\ j^2-4k=-d}} \prod_{\ell | k} \left(1 + \frac{1}{\ell}\right).
\]
First, note that
\begin{align*}
\prod_{\ell | k,\,\ell>\sqrt{\log K}} \left(1 + \frac{1}{\ell}\right) & \ll  \prod_{\ell | k,\,\ell>\log K} \left(1 + \frac{1}{\ell}\right)
\ll \exp\left\{ \sum_{\ell|k,\,\ell>\log K}\frac1{\ell} \right\} \le \exp\left\{ \frac{\#\{\ell|k\}}{\log K} \right\} \ll 1,
\end{align*}
by Mertens's estimate and the fact that $k$ has at most $\frac{\log k}{\log 2}$ distinct prime factors. Therefore,
\begin{align*}
\prod_{\ell | k} \left(1 + \frac{1}{\ell}\right) &\ll \prod_{\ell | k,\,\ell\le\sqrt{\log K}} \left(1 + \frac{1}{\ell}\right)
=\sum_{ \substack{ a | k \\ P^+(a)\le\sqrt{\log K} }}\frac{\mu^2(a)}a.
\end{align*}
So
\begin{align*}
T_d &\ll \sum_{P^+(a)\le\sqrt{\log K}} \frac{\mu^2(a)}a \sum_{\substack{k\le K,\, |j| < 2\sqrt{k}\\ a|k,\,j^2-4k=-d}}1
\le\sum_{P^+(a)\le(\log K)^{1/2}}\frac{\mu^2(a)}a\sum_{\substack{|j| < 2\sqrt{K}\\4a|j^2+d}}1\\
&\ll\sum_{P^+(a)\le(\log K)^{1/2}}\frac{\mu^2(a)}a\cdot \tau(a)\left(\frac{\sqrt{K}}a+1\right) \ll\sqrt{K},
\end{align*}
since $a\le e^{\pi(\sqrt{\log K})}\ll \sqrt{K}$ for all square-free integers $a$ with $P^+(a)\le\sqrt{\log K}$. Consequently,
\[
\#S(M,K) \ll \frac{M\sqrt{K}}{\log M}\sum_{d\le4K}\prod_{\ell\le y}\left(1-\frac{\leg{-d}{\ell}}\ell\right)+M^{1/2}K^{3/2}\log M
\]
Using Lemma \ref{lemma-Elliot} on truncated products of $L$-functions with $\delta=1/4$ and $Q=4K$, we find that there is a set $\mathcal{E}$ of $O(K^{1/4})$ integers in $[1,4K]$ such that if $d\in[1,4K]$ and the conductor of $\leg{-d}{\cdot}$ is not in $\mathcal{E}$, then
\[
\prod_{w_1<\ell\le w_2}\left(1-\frac{\leg{-d}{\ell}}{\ell}\right)\asymp 1\quad(w_2\ge w_1\ge \sqrt{\log(4K) }).
\]
So for such a $d$ we find that
\begin{equation}\label{shortproduct}
\prod_{\ell\le y}\left(1-\frac{\leg{-d}{\ell}}{\ell}\right) \asymp \prod_{\ell\le z}\left(1-\frac{\leg{-d}{\ell}}{\ell}\right),
\end{equation}
where $z=\min\{y,\sqrt{\log(4K)}\}$. For the exceptional $d$'s, we write $-d = -a^2d_1$, where $d_1$ denotes the conductor of $\leg{-d}{\cdot}$ and note that
\[
\prod_{\ell\le y}\left(1-\frac{\leg{-d}{\ell}}\ell \right) \ll \frac{a}{\phi(a)} \prod_{\ell\le y}\left(1-\frac{\leg{-d_1}\ell}{\ell} \right)
\ll \frac{a}{\phi(a)}  \; |d_1|^{3/4}
\]
by Lemma \ref{lemma-truncated}. Hence,
\begin{align*}
\sum_{\substack{d\le4K \\ {\rm cond}\left(\leg{-d}{\cdot}\right)\in\mathcal{E}}} \prod_{\ell\le y}\left(1-\frac{\leg{-d}{\ell}}\ell\right)
& \le \sum_{d_1\in\mathcal{E}} \sum_{1\le|a|\le\sqrt{4K/|d_1|}} \prod_{\ell \leq y} \left(1-\frac{\leg{-d_1a^2}{\ell}}{\ell}\right)\\
&\ll\sum_{d_1\in\mathcal{E}} \sum_{1\le |a|\le\sqrt{4K/|d_1|}} \frac{a}{\phi(a)} \; |d_1|^{3/4}\\
&\ll \sum_{d_1\in\mathcal{E}} |d_1|^{1/4}\sqrt{K}\ll K^{1/4} \cdot K^{1/4}\cdot\sqrt{K}=K.
\end{align*}
The above relation and~\eqref{shortproduct} then imply that
\begin{equation}\label{beforelast}
\#S(M,K) \ll \frac{M\sqrt{K}}{\log M}\sum_{ d \le 4K }\prod_{\ell \le z}
\left(1-\frac{\leg{-d}{\ell}}{\ell} \right)+\frac{M K^{3/2}}{\log M}.
\end{equation}
In order to control the above sum, we proceed by expanding the product to a sum and inverting the order of summation.
We have that
\begin{align*}
\sum_{d\le4K}\prod_{\ell \le z} \left(1-\frac{\leg{-d}{\ell}}{\ell}\right)
&=\sum_{P^+(a)\le z} \frac{\mu(a)}a \sum_{d\le 4K}\left(\frac{-d}a\right).
\end{align*}
If $a=1$, the inner sum is $4K+O(1)$; else, it is $\ll a$. So
\begin{align*}
\sum_{d\le4K}\prod_{\ell\le z}\left(1-\frac{\left(\frac{-d}{\ell}\right)}{\ell}\right)
&\ll K+\sum_{P^+(a)\le z}\mu^2(a) = K+2^{\pi(z)} \le K + 2^{\pi(\sqrt{\log(4K)})} \ll K.
\end{align*}
Inserting the last estimate into~\eqref{beforelast}, we obtain the inequality
\[
\#S(M,K) \ll \frac{MK^{3/2}}{\log M},
\]
which completes the proof of Theorem~\ref{thm-ksmall}.

%%%%%%%%%%%%%%%%%%%%%%%%%%%%%%%%%%%%%%%%%%%%%%%%%%%%%%%%%%%%%%%%%%%%%%%%%%%%%%%%%%%%%%%%%%%%%%%%%%%%%%%%%%%%%%%%%%%%%%%%%
%%%%%%%%%%%%%%%%%%%%%%%%%%%%%%%%%%%%%%%%%%%%%%%%%%%%%%%%%%%%%%%%%%%%%%%%%%%%%%%%%%%%%%%%%%%%%%%%%%%%%%%%%%%%%%%%%%%%%%%%%
%%%%%%%%%%%%%%%%%%%%%%%%%%%%%%%%%%%%%%%%%%%%%%%%%%%%%%%%%%%%%%%%%%%%%%%%%%%%%%%%%%%%%%%%%%%%%%%%%%%%%%%%%%%%%%%%%%%%%%%%%
%%%%%%%%%%%%%%%%%%%%%%%%%%%%%%%%%%%%%%%%%%%%%%%%%%%%%%%%%%%%%%%%%%%%%%%%%%%%%%%%%%%%%%%%%%%%%%%%%%%%%%%%%%%%%%%%%%%%%%%%%
%%%%%%%%%%%%%%%%%%%%%%%%%%%%%%%%%%%%%%%%%%%%%%%%%%%%%%%%%%%%%%%%%%%%%%%%%%%%%%%%%%%%%%%%%%%%%%%%%%%%%%%%%%%%%%%%%%%%%%%%%

\section{Proof of Theorem \ref{Kasymp}}

Define
\[
R(M,K) = \{ M/2<m\le M,\,K/2<k\le K: \text{there is no prime}\ p\equiv 1\pmod m\ \text{in}\ I_{m^2k} \}.
\]
First, we prove an intermediate result for the cardinality of $R(M,K)$.

\begin{thm}\label{Kasymp_v1} Fix $A\ge1$ and $\epsilon\in(0,1/6]$. If $M \leq K^{1/4-\epsilon}$, then
\[
 \# R(M, K) \ll_{\epsilon,A} \frac{MK}{(\log K)^A}.
\]
If, in addition, the Riemann hypothesis for Dirichlet $L$-functions is true, then the above estimate holds when $M\le K^{1/2-\epsilon}$.\end{thm}

\begin{proof} We prove both parts of the theorem simultaneously. Set $h=M\sqrt{K}$ and
\[
E(y,h;q,a) = \left| \psi(y+h;q,a) - \psi(y;q,a) - \frac{h}{\phi(q)}\right|,
\]
and note that if the pair $(m,k)\in R(M,K)$, then
\[E((m\sqrt{k}-1)^2,h;m,1) \gg \frac{h}{\phi(m)} \geq \frac{h}{m} \asymp \sqrt{K}.
\]
Consequently,
\begin{align*}
\# R(M,K)
 &\ll \frac{1}{\sqrt{K}}\sum_{M/2<m\le M}\sum_{K/2<k\le K} E((m\sqrt{k}-1)^2,h;m,1).
\end{align*}
Next, observe that $(m\sqrt{k}-1)^2\in J:=[M^2K/10,M^2K]$. We cover the interval $J$ by $O((M^2K)^{1-\lambda})$ subintervals $J_r$ of length $(M^2K)^{\lambda}$ each, where $\lambda\in[1/4,1/2)$ is a fixed parameter to be chosen later. If $(m\sqrt{k}-1)^2\in J_r$, then for every $y\in J_r$ we have that
\begin{align*}
\left| E(y,h;m,1)-E((m\sqrt{k}-1)^2,h;m,1) \right|
	&\le  \#\{ n \in J_r\cup(J_r+h): n \equiv 1\pmod m\} \\
	&\ll 1+ \frac{(M^2K)^{\lambda}}{m}
		\asymp  \frac{(M^2K)^{\lambda}}{M} ,
\end{align*}
since $M\le K^{1/2}$ and $\lambda\ge1/4$. So if we let $\meas(J_r)$ denote the length of the interval $J_r$, then
\begin{align*}
E((m\sqrt{k}-1)^2,h;m,1)
	&= \frac1{\meas(J_r)}\int_{J_r} E(y,h;m,1)dy+O\left( \frac{(M^2K)^{\lambda}}{M} \right) \\
	& \ll \frac1{(M^2K)^{\lambda}} \int_{J_r} E(y,h;m,1)dy + \frac{(M^2K)^{\lambda} }{M},
\end{align*}
and consequently
\begin{align*}
\# R(M,K)
	 &\ll \frac{1}{\sqrt{K}}\sum_{ M/2<m\le M}\sum_{J_r}\left(\frac1{(M^2K)^{\lambda}} \int_{J_r} E(y,h;m,1)dy
 		+ \frac{(M^2K)^{\lambda}}{M}\right)
			 \sum_{\substack{ K/2<k\le K \\ (m \sqrt{k} - 1)^2 \in J_r}} 1.
\end{align*}
For every fixed $m\in[M/2,M]$ and every fixed interval $J_r$, there are at most
$1+O((M^2K)^{\lambda}/M^2)$ values of $k$ with $(m\sqrt{k}-1)^2\in J_r$. Since
$$ \frac{(M^2K)^\lambda}{M^2} \gg 1 \iff M \ll K^{\lambda/(2 - 2 \lambda)},$$
by choosing  $\lambda \in [1/4, 1/2)$ appropriately in terms of $\epsilon$, we get
that $M\le K^{1/2-\epsilon}$ implies that
$1+O((M^2K)^{\lambda}/M^2) = O((M^2K)^{\lambda}/M^2)$.  Therefore, there are at most
$O((M^2K)^{\lambda}/M^2)$ values of $k$ with $(m\sqrt{k}-1)^2\in J_r$, and we deduce that
\begin{align}
\# R(M,K)
 &\ll \frac{1}{\sqrt{K}}\sum_{M/2<m\le M}  \sum_{J_r}  \left(\frac1{(M^2K)^{\lambda}} \int_{J_r} E(y,h;m,1)dy
 		+ \frac{(M^2K)^{\lambda}}{M}\right)	\frac{(M^2K)^{\lambda}}{M^2}  \nonumber \\
 &\le \frac{1}{M^2\sqrt{K}}\sum_{M/2<m\le M}  \int_{M^2K/20}^{M^2K}E(y,h;m,1)dy + O(M^{2\lambda} K^{1/2+\lambda} ). 	
\label{before-bv}
\end{align}
If $M \leq K^{1/4 - \epsilon}$, then we have that
\[
M^2 \leq \frac{M\sqrt{K}}{(M^2K)^{1/6+\epsilon/2}} = \frac{h}{(M^2K)^{1/6+\epsilon/2}},
\]
and we can apply the Lemma~\ref{average bv} to get that
\begin{equation}\label{bv-short-intervals}
\sum_{M/2<m\le M}  \int_{M^2K/20}^{M^2K}E(y,h;m,1)dy \ll \frac{M^2K h}{(\log K)^A} = \frac{M^{3} K^{3/2}}{(\log K)^A}.
\end{equation}
Similarly, if $M \leq K^{1/2 - \epsilon}$, then
\[
M^2 \leq \frac{M\sqrt{K}}{(M^2K)^{\epsilon/2}} = \frac{h}{(M^2K)^{\epsilon/2}} .
\]
So, if the Riemann hypothesis for Dirichlet $L$-functions holds, then Lemma \ref{average bv} implies that \eqref{bv-short-intervals} holds in this case too. Inserting this relation into \eqref{before-bv}, we deduce that
\[
\# R(M,K)\ll_{\epsilon,A} \frac{MK}{(\log K)^{A}} + M^{2\lambda} K^{1/2+\lambda} \ll_{\epsilon,A} \frac{MK}{(\log K)^A}
\]
since $\lambda<1/2$ and $M\le K^{1/2-\epsilon}$.
This completes the proof of Theorem \ref{Kasymp_v1}.
\end{proof}

\begin{proof}[Proof of Theorem~\ref{Kasymp}] Let $\epsilon$, $M$ and $K$ be as in the statement of Theorem \ref{Kasymp}. Clearly
\begin{align*}
 \#\{ m\le M,\,k\le K :\text{there is no prime}\ p\equiv 1\pmod m\ \text{in}\ I_{m^2k} \} & \le \sum_{2^a\le 2M,\,2^b\le 2K} \# R(2^a,2^b).
\end{align*}
If $2^b\le K^{1-\epsilon}$, then we use the trivial bound $\# R(2^a,2^b)\le 2^{a+b}$. Otherwise, we have that
\[
2^{a-1}\le M\le
	\begin{cases}
		K^{1/4-\epsilon}\le 2^{b(1/4-\epsilon/2)}   &\text{if}\ M\le K^{1/4-\epsilon},\cr
		K^{1/2-\epsilon}\le 2^{b(1/2-\epsilon/2)}   &\text{if}\ M\le K^{1/2-\epsilon},
	\end{cases}
\]	
and so Theorem~\ref{Kasymp_v1} implies that $\#R(2^a,2^b)\ll_{\epsilon,A} 2^{a+b}/b^A$. Consequently,
\begin{align*}
& \# \{ m\le M,\,k\le K :\text{there is no prime}\ p\equiv 1\pmod m\ \text{in}\ I_{m^2k} \} \\
 &\quad \ll_A \sum_{ \substack{ 2^a\le 2M \\ 2^b\le K^{1-\epsilon} }}2^{a+b}+ \sum_{\substack{ 2^a\le 2M \\ K^{1-\epsilon}< 2^b\le 2K }} \frac{2^{a+b}}{b^A} \ll_{\epsilon,A}\frac{MK}{(\log K)^A}.
\end{align*}
The above estimate and Remark \ref{rem:aftercor2.2} complete the proof of Theorem~\ref{Kasymp}.
\end{proof}

%%%%%%%%%%%%%%%%%%%%%%%%%%%%%%%%%%%%%%%%%%%%%%%%%%%%%%%%%%%%%%%%%%%%%%%%%%%%%%%%%%%%%%%%%%%%%%%%%%%%%%%%%%%%%%%%%%%%%%%%%
%%%%%%%%%%%%%%%%%%%%%%%%%%%%%%%%%%%%%%%%%%%%%%%%%%%%%%%%%%%%%%%%%%%%%%%%%%%%%%%%%%%%%%%%%%%%%%%%%%%%%%%%%%%%%%%%%%%%%%%%%
%%%%%%%%%%%%%%%%%%%%%%%%%%%%%%%%%%%%%%%%%%%%%%%%%%%%%%%%%%%%%%%%%%%%%%%%%%%%%%%%%%%%%%%%%%%%%%%%%%%%%%%%%%%%%%%%%%%%%%%%%
%%%%%%%%%%%%%%%%%%%%%%%%%%%%%%%%%%%%%%%%%%%%%%%%%%%%%%%%%%%%%%%%%%%%%%%%%%%%%%%%%%%%%%%%%%%%%%%%%%%%%%%%%%%%%%%%%%%%%%%%%
%%%%%%%%%%%%%%%%%%%%%%%%%%%%%%%%%%%%%%%%%%%%%%%%%%%%%%%%%%%%%%%%%%%%%%%%%%%%%%%%%%%%%%%%%%%%%%%%%%%%%%%%%%%%%%%%%%%%%%%%%

\section{Proof of Theorem \ref{Kbiglower}}

Note that, as a direct consequence of Corollary \ref{basicprop} and the fact that $I_1=(0,4)$, the primes 2 and 3 are always contained in $S(M,K)$. In particular,  $\# S(M,K) \ge 2$ and hence we may assume without loss of generality that $K$ is large enough. Also, we may assume that $M$ is an integer, so that the interval $(3M/4,M]$ always contain integers.

From Remark \ref{rem:aftercor2.2}, we have that
$$
\# S(M, K) = \sum_{m \leq M} \sum_{k \leq K} \mathbb{I} (m,k).
$$
Now, note that
$$
\mathbb{I} (m,k)  \ge \frac{\phi(m) \log(4\sqrt{k})}{8m\sqrt{k}} \cdot \# \left\{ p \in I_{m^2k} : p \equiv 1 \pmod{m} \right\}
$$
by the Brun-Titchmarsh inequality. Therefore we deduce that
\begin{align}
\# S(M, K) &\gg \frac{\log K}{\sqrt K}\sum_{ \substack{ 3M/4< m \leq M \\ K/5 < k \leq K }} \frac{\phi(m)}{m}
\sum_{ \substack{p \in I_{m^2k} \\ p \equiv 1 \pmod m}} 1 \nonumber\\
&\ge \frac{\log{K}}{\sqrt{K}} \sum_{M^2K/4< p < M^2K/3} \sum_{\substack{ 3M/4 < m \leq M \\ p \equiv 1 \pmod m}} \frac{\phi(m)}{m}
\sum_{ \substack{ K/5 < k \leq K \\ p \in I_{m^2k}}} 1 \label{Kbiglower firstinequality}
\end{align}
by switching the order of summation and restricting $p$ in the interval $(M^2K/4,M^2K/3]$. Fix $p$ and $m$ as in~\eqref{Kbiglower firstinequality} and note that if $k$ is an integer for which $p\in I_{m^2k}$, then we necessarily have that $K/5<k\le K$. So
$$
\sum_{\substack{K/5 < k \leq K \\ p \in I_{m^2k}}} 1 = \sum_{k\in\SZ\,:\,p \in I_{m^2k} } 1 =
\#\left\{ k\in \Z: \frac{p-2\sqrt{p}+1}{m^2} < k < \frac{p+2\sqrt{p}+1}{m^2} \right\} \gg \frac{\sqrt{p}}{m^2} \asymp \frac{\sqrt{K}}M,
$$
since
$$
\frac{4\sqrt{p}}{m^2}> \frac{4\sqrt{M^2K/4}}{M^2} = \frac{2\sqrt{K}}{M}\ge 2
$$
by our assumption that $M\le \sqrt{K}$. Consequently,
\begin{align}
\# S(M, K) &\gg \frac{\log K}{M} \sum_{M^2K/4 < p < M^2K/3} \sum_{\substack{ 3M/4 < m \le M  \\ p\equiv 1\pmod m}} \frac{\phi(m)}{m} \nonumber\\
& = \frac{\log K}{M} \sum_ { 3M/4 < m\le M} \frac{\phi(m)}{m} \left(
\pi(M^2K/3; m, 1)-\pi(M^2K/4;m,1) \right) \nonumber\\
& = \frac{\log K}{M}\left( \sum_ { 3M/4 < m\le M} \frac{\phi(m)}{m} \frac{{\rm li}(M^2K/3)-{\rm li}(M^2K/4)}{\phi(m)} + E\right), \label{before bv}
\end{align}
where
\[
E =  \sum_{3M/4<m\le M} \frac{\phi(m)}{m} \left( \pi(M^2K/3; m, 1)-\pi(M^2K/4;m,1) - \frac{{\rm li}(M^2K/3)-{\rm li}(M^2K/4)}{\phi(m)} \right) .
\]
The sum over $m$ in \eqref{before bv} is
\[
\sum_ { 3M/4 < m\le M} \frac{{\rm li}(M^2K/3)-{\rm li}(M^2K/4)}{m} \gg \sum_{ 3M/4 < m\le M} \frac{M^2K}{M\log(M^2K)} \asymp \frac{M^2K}{\log K}.
\]
It is in this step that we need that the interval $(3M/4,M]$ contains an integer.
Furthermore, we have that
\begin{align*}
|E| &\le \sum_{3M/4<m\le M} \left|\pi(M^2K/4; m, 1)-\pi(M^2K/3;m,1) - \frac{{\rm li}(M^2K/3)-{\rm li}(M^2K/4)}{\phi(m)} \right| \\
&\ll \frac{M^2K}{\log^2(M^2K)},
\end{align*}
by Lemma~\ref{lemma-BVthm}. Combining the above estimates, we find that there is an absolute constant $c$ such that
$$
\# S(M,K) \ge cMK+O\left(\frac{MK}{\log K}\right) \ge \frac{cMK}{2},
$$
provided that $K$ is large enough. This completes the proof of Theorem~\ref{Kbiglower}.

\section*{Acknowledgements} The authors would like to thank Roger Heath-Brown for suggesting some useful references about counting primes in short intervals. The second author would also like to thank Arul Shankar for enlightening discussions about the Cohen-Lenstra heuristics in the context of elliptic curves and abelian varieties over finite fields.
The work of the second author was supported by the Natural Sciences and Engineering Research Council of Canada [Discovery Grant 155635-2008]. The first, third, and fourth authors were supported by postdoctoral fellowships from the Centre de recherches math\'ematiques at Montr\'eal during the completion of most of this work. They are grateful for the financial support and the pleasant working environment.

\end{document}